\documentclass[12pt]{amsart}
\usepackage[margin=1in,letterpaper,portrait]{geometry}
 
\usepackage{amsmath,amssymb,amsthm}
\usepackage{verbatim} 
\usepackage{url}
\usepackage{relsize}
\usepackage{enumerate}
\usepackage[colorlinks=true,linkcolor=blue,citecolor=magenta]{hyperref} 
\newcommand{\mb}{\mathbb}

\newcommand{\gauss}[2]{\genfrac{[}{]}{0pt}{}{#1}{#2}_q}
  
 \theoremstyle{definition}
\newtheorem{Thm}{Theorem}[section] 
\newtheorem{Cor}[Thm]{Corollary}

\newtheorem{Lem}[Thm]{Lemma} 
\newtheorem{Prop}[Thm]{Proposition}
\newtheorem{Def}[Thm]{Definition}

\newtheorem{Rem}[Thm]{Remark}

\def\F{{\mathbb F}}

\def\Fq{{\mathbb F}_q}

\allowdisplaybreaks
\makeatletter
\def\imod#1{\allowbreak\mkern10mu({\operator@font mod}\,\,#1)} 
\makeatother

\title{Splitting subspaces of Linear Operators over Finite Fields} 

\author{Divya Aggarwal} 
\address{Indraprastha Institute of Information Technology Delhi (IIIT-Delhi), New Delhi 110020, India.}
\email{divyaa@iiitd.ac.in} 

\author{Samrith Ram}
\address{Indraprastha Institute of Information Technology Delhi (IIIT-Delhi), New Delhi 110020, India.}
\email{samrith@gmail.com}



\begin{document}

\begin{abstract}
Let $V$ be a vector space of dimension $N$ over the finite field $\Fq$ and $T$ be a linear operator on $V$. Given an integer $m$ that divides $N$, an $m$-dimensional subspace $W$ of $V$ is $T$-splitting if $V=W\oplus TW\oplus \cdots \oplus T^{d-1}W$ where $d=N/m$. Let $\sigma(m,d;T)$ denote the number of $m$-dimensional $T$-splitting subspaces. Determining $\sigma(m,d;T)$ for an arbitrary operator $T$ is an open problem. We prove that $\sigma(m,d;T)$ depends only on the similarity class type of $T$ and give an explicit formula in the special case where $T$ is cyclic and nilpotent. Denote by $\sigma_q(m,d;\tau)$ the number of $m$-dimensional splitting subspaces for a linear operator of similarity class type $\tau$ over an $\Fq$-vector space of dimension $md$. For fixed values of $m,d$ and $\tau$, we show that $\sigma_q(m,d;\tau)$ is a polynomial in $q$.
\end{abstract} 
\maketitle
\tableofcontents

\section{Introduction}



Let $\Fq$ denote the finite field with $q$ elements where $q$ is a prime power. Throughout this paper, $V$ will denote a finite-dimensional vector space over $\Fq$. The variables $m$ and $d$ will always denote positive integers. Let $V$ be an $md$-dimensional vector space over the finite field $\Fq$. We begin with a definition.

\begin{Def}
Let $T$ be a linear operator on $V$. A subspace $W$ of $V$ of dimension $m$ is a \emph{splitting subspace} for $T$ if 
\begin{align*}
V=W\oplus T W\oplus \cdots\oplus T^{d-1}W.
\end{align*}
\end{Def} 
The above definition was motivated by the following question asked by Niederreiter \cite[p. 11]{N2}: Let $\alpha \in \F_{q^{md}}$ such that $\alpha$ is a generator of the cyclic group $\F_{q^{md}}^*$ of nonzero elements in $\F_{q^{md}}$. How many $m$-dimensional $\Fq$-linear subspaces $W$ of $\F_{q^{md}}$ satisfy
$$
\F_{q^{md}}=W\oplus \alpha W\oplus \cdots\oplus \alpha^{d-1}W?
$$
Niederreiter encountered this problem in the context of his work on the multiple-recursive matrix method for pseudorandom number generation. This question was settled by Chen and Tseng~\cite[Cor. 3.4]{sscffa} who proved a conjecture \cite[Conj. 5.5]{m=2} that the number of such subspaces is 
\begin{equation}
  \label{eq:split}
\frac{q^{md}-1}{q^m-1}q^{m(m-1)(d-1)}.  
\end{equation}
Another proof of the result of Chen and Tseng and some connections with unimodular matrices may be found in \cite{arora2020unimodular}. Given a linear operator $T$ on $V$, let $\sigma(m,d;T)$ denote the number of $m$-dimensional $T$-splitting subspaces. Finding a formula for  $\sigma(m,d;T)$ for arbitrary $T$ is an open problem \cite[p. 54]{split}. If $T$ has an irreducible characteristic polynomial, then it follows from the work of Chen and Tseng that $\sigma(m,d;T)$ is given by the expression in \eqref{eq:split} above.   
 
In fact, splitting subspaces are closely related to anti-invariant subspaces. A subspace $W$ is said to be $k$-fold anti-invariant if the sum
$$
W+ TW+ \cdots + T^kW
$$
is direct. Barría and Halmos \cite{MR748946} and Sourour \cite{MR822138} determined the maximal dimension of a 1-fold anti-invariant subspace for a given operator $T$. These results were generalized by Knüppel and Nielsen \cite{MR2013452} who solved the same problem for $k$-fold anti-invariant subspaces.

Splitting subspaces also arise in a slightly different context in the setting of finite fields. Suppose $S=\{\alpha_1,\ldots,\alpha_m\}$ is a subset of $V$ and let $W_S$ denote the linear span of $S$. The Krylov subspace of order $d$ generated by $S$ is defined by
$$
\mathrm{Kry}(T,S;d):=W_S +TW_S+\cdots+ T^{d-1}W_S.
$$
Let $\kappa_{m,d}(T)$ denote the probability that the Krylov subspace of order $d$ for a randomly chosen ordered subset $S$ of cardinality $m$ is all of $V$. Since $\mathrm{Kry}(T,S;d)=V$ precisely when $W_S$ is a splitting subspace for $T$, it follows that
$$
\kappa_{m,d}(T)=\frac{\gamma_m(q)\cdot \sigma(m,d;T)}{q^{m^2 d}},  
$$ 
where $\gamma_m(q)=(q^m-1) \ldots (q^m-q^{m-1})$ denotes the number of ordered bases of an $m$-dimensional vector space over $\Fq$. The problem of solving large sparse linear systems over finite fields arises in computer algebra and number theory. Block iterative methods such as the Wiedemann algorithm, which are based on finding linear relations in Krylov subspaces, are used to solve such systems. Giving bounds on the probability $\kappa_{m,d}(T)$ is a difficult and important problem \cite[p. 2]{Brentetal2003} in the analysis of such algorithms. We refer to Brent, Gao and Lauder~\cite{Brentetal2003} and the references therein for more on this topic.

In this paper, we are mainly interested in determining $\sigma(m,d;T)$. Building upon earlier work by Chen and Tseng, we give a general recurrence which may be solved to obtain an expression for the number of splitting subspaces. We prove (Corollary~\ref{type}) that for fixed values of $m$ and $d$, the number $\sigma(m,d;T)$ depends only on the similarity class type (see Definition~\ref{def:sct}) of $T$. A crucial ingredient in the proof is a theorem of Brickman and Fillmore~\cite{MR213378} on the structure of the invariant subspace lattice of a primary transformation.  Let $\sigma_q(m,d;\tau)$ denote the number of $m$-dimensional splitting subspaces for a linear operator of similarity class type $\tau$ defined over an $\Fq$-vector space of dimension $md$. 
In Section~\ref{sec:cyclic}, we show that there is a simple formula for the number of splitting subspaces of a cyclic operator that is either nilpotent or unipotent. We also give an application of our results to the enumeration of invertible matrices having a special form. Finally, we show that, for $m,d,\tau$ fixed, the quantity $\sigma_q(m,d;\tau)$ is a polynomial in $q$.  
\section{Counting Flags of Subspaces} 
\label{sec:recurs}
To unravel the enumeration problem, we begin with some definitions and notation introduced by Chen and Tseng \cite[Sec. 2]{sscffa} to count a more general class of subspaces that includes the $T$-splitting subspaces. In what follows, $V$ denotes a vector space of dimension $N$ over the finite field $\mb{F}_{q}$ and $T$ a linear operator on $V$. 

\begin{Def}
Suppose $S_{1},S_{2},\ldots,S_{k}$ are sets of subspaces of $V$.
Let $[S_{1},S_{2},\ldots,S_{k}]_{T}$ denote the set of all $k$-tuples
$(W_{1},W_{2},\ldots,W_{k})$
such that
\begin{alignat*}{2}
W_{i}\in S_{i}&\quad\text{for}&\ 1\leq &\ i\leq k,
\\  W_{i} \supseteq W_{i+1}+T W_{i+1} &\quad\text{for}&\quad\ 1\leq &\ i\leq k-1.
\end{alignat*}
\end{Def}
If $S_{i}$ is the set of all subspaces of $V$ of dimension $d_{i}$ for some $i$, then $S_{i}$ is denoted within the brackets as $d_{i}$. For instance, $[5,3]_T$ denotes the set all tuples $(W_{1},W_{2})$ such that $\dim(W_{1})=5$, $\dim(W_{2})=3$  and $ W_1\supseteq W_{2}+T W_{2}$.

\begin{Def}
Let $a$, $b$ be nonnegative integers such that $N\geq a\geq b$. Define
\begin{align*}
(a,b)_T=\left\{W\subseteq V: \dim(W)=a\ \text{and}\ \dim(W\cap T^{-1} W)=b\right\}.
\end{align*}
\end{Def}
For instance, $(3,2)_T$ denotes the set of $3$-dimensional subspaces $W$ for which $\dim(W\cap T^{-1} W)=2$.
Note that $(a,a)_T$ is the set of all subspaces of dimension $a$ which are invariant under $T$. We will freely use $[S_{1},S_{2},\ldots,S_{k}]$ to denote $[S_{1},S_{2},\ldots,S_{k}]_{T}$ and $(a,b)$ to denote $(a,b)_T$ when there is only one operator or when the operator under consideration is clear from the context.

\begin{Def}
Suppose $[S_{1,1},S_{1,2}],[S_{2,1},S_{2,2}],\ldots,[S_{r,1},S_{r,2}]$ are sets as defined above. Then
\begin{align*}
\left\langle[S_{1,1},S_{1,2}],[S_{2,1},S_{2,2}],\ldots,[S_{r,1},S_{r,2}]\right\rangle
\end{align*} 
denotes the set of $2r$-tuples of subspaces $(W_{1,1},W_{1,2},W_{2,1},W_{2,2},\ldots, W_{r,1},W_{r,2})$ such that
\begin{alignat*}{2}
(W_{i,1},W_{i,2})\in[S_{i,1},S_{i,2}]&\quad\text{for}&\quad 1\leq &\ i\leq r,
\\W_{i,2} \supseteq W_{i+1,1} &\quad\text{for}&\quad 1\leq &\ i\leq r-1.
\end{alignat*}
\end{Def}

For instance, $\left\langle [5,4],[3,2]\right\rangle$ is the set of all $4$-tuples of subspaces $(W_{1},W_{2},W_{3},W_{4})$ such that
\begin{alignat*}{2}
\dim(W_{1})=5,\ \dim(W_{2})=4,\ \dim(W_{3})=3,\ \dim(W_{4})=2,
\\ W_1 \supseteq W_{2}+T W_{2} ,\quad W_3\supseteq W_{4}+T W_{4} ,
\quad W_{2} \supseteq W_{3}.
\end{alignat*}

The next definition specifies an ordering on tuples labelling sets of the form
\begin{align*}
[(a_{1,1},a_{1,2}),(a_{2,1},a_{2,2}),\ldots,(a_{r,1},a_{r,2})].
\end{align*}
\begin{Def}
  \label{def:order} 
Define an ordering on ordered pairs $(a,b)$ such that $(a_{1},b_{1})\succeq (a_{2},b_{2})$ if $a_{1}>a_{2}$ or $a_{1}=a_{2}$ and $b_{1}\leq b_{2}$. Extend the ordering to tuples of the form $$[(a_{1,1},a_{1,2}),(a_{2,1},a_{2,2}),\ldots,(a_{r,1},a_{r,2})]$$ in such a way that the order is lexicographic in terms of the ordered pairs $(a_{i,1},a_{i,2})$ from left to right. 

For instance, $(5,2)\succ(5,3)\succ(2,1)$ while $[(8,6),(5,3)]\succ[(8,6),(5,4)]\succ[(7,6),(6,2)]$.
\end{Def}

For fixed $r$, the ordering $\succeq$ is a total order. The following proposition is used repeatedly in constructing the recursion and follows easily from the definitions above.
\begin{Prop}
\label{expand}
For nonnegative integers $N\geq a\geq b$, we have
\begin{align*}
[a,b]&=\bigcup_{i=b}^{a}{[(a,i),b]}
\\&=\bigcup_{j=0}^{b}{[a,(b,j)]}.
\end{align*}
\end{Prop}


The recursion in the next lemma expresses the cardinality of sets of subspaces labelled by a tuple $\nu$ in terms of the cardinality of sets labelled by tuples $\mu\prec \nu$ in the ordering. The base cases are of the form $\left| [(a_{1},a_{1}),(a_{2},a_{2}),\ldots,(a_{r},a_{r})]_T\right|$. 
 The following lemma and Proposition~\ref{splittingsubspaces} extend results obtained by Chen and Tseng that hold for invertible operators.  

\begin{Lem}
\label{recursion}
Let $T$ be a linear operator on an $N$-dimensional vector space. Suppose
\begin{align*}
a_{0,1}=a_{0,2}=N\geq a_{1,1}\geq a_{1,2}\geq a_{2,1}\geq a_{2,2}\geq\ldots\geq a_{r,1}\geq a_{r,2}\geq 0=a_{r+1,1}=a_{r+1,2}\\
 \mbox{ and } a_{i-1,1}\geq 2a_{i,1}-a_{i,2}\quad\text{for}\quad 1\leq i\leq r.
\\(\mbox{If the conditions are not satisfied, then}~[(a_{1,1},a_{1,2}),(a_{2,1},a_{2,2}),\ldots,(a_{r,1},a_{r,2})] \mbox{ is empty}).
\end{align*}

Further, let
\begin{align*}
&A=\left\{(j_{1},\ldots,j_{r}):\max(a_{i+1,2},2a_{i,2}-a_{i,1})\leq j_{i}\leq a_{i,2} \mbox{ and } 1\leq i\leq r\right\},
\\&B=\left\{(k_{1},\ldots,k_{r}):a_{i,2}\leq k_{i}\leq a_{i,1}\mbox{ and } 1\leq i\leq r\right\}.
\end{align*}
 
Then
\begin{align*}
&|[(a_{1,1},a_{1,2}),(a_{2,1},a_{2,2}),\ldots,(a_{r,1},a_{r,2})]|
\\&=\sum_{(j_{1},\ldots,j_{r})\in A}{|[(a_{1,2},j_{1}),(a_{2,2},j_{2}),\ldots,(a_{r,2},j_{r})]|\prod_{i=1}^{r}{\gauss{a_{i-1,2}-(2a_{i,2}-j_{i})}{a_{i,1}-(2a_{i,2}-j_{i})}}}
\\ &-\sum_{(k_{1},\ldots,k_{r})\in B\backslash (a_{1,2},\ldots,a_{r,2})} {|[(a_{1,1},k_{1}),(a_{2,1},k_{2}),\ldots,(a_{r,1},k_{r})]|\prod_{i=1}^{r}{\gauss{k_{i}-a_{i+1,1}}{a_{i,2}-a_{i+1,1}}}}.
\end{align*}
\end{Lem}
\begin{proof}
The proof is along the lines of \cite[Lem. 2.7]{sscffa}. The size of
$[(a_{1,1},a_{1,2}),(a_{2,1},a_{2,2}),\ldots,(a_{r,1},a_{r,2})]$ is computed by applying Proposition~\ref{expand}. Consider
\begin{align*}
&|\left\langle[a_{1,1},a_{1,2}],[a_{2,1},a_{2,2}],\ldots,[a_{r,1},a_{r,2}]\right\rangle|
\\ &=\sum_{(k_{1},\ldots,k_{r})\in B} {|\left\langle[(a_{1,1},k_{1}),a_{1,2}],[(a_{2,1},k_{2}),a_{2,2}],\ldots,[(a_{r,1},k_{r}),a_{r,2}]\right\rangle|}
\\&=\sum_{(k_{1},\ldots,k_{r})\in B} {|[(a_{1,1},k_{1}),(a_{2,1},k_{2}),\ldots,(a_{r,1},k_{r})]|\prod_{i=1}^{r}{\gauss{k_{i}-a_{i+1,1}}{a_{i,2}-a_{i+1,1}}}}\label{right}\tag{R}.
\end{align*}
Using Proposition~\ref{expand} again, we have
\begin{align*}
&|\left\langle[a_{1,1},a_{1,2}],[a_{2,1},a_{2,2}],\ldots,[a_{r,1},a_{r,2}]\right\rangle|
\\&=\sum_{(j_{1},\ldots,j_{r})\in A}{|\left\langle[a_{1,1},(a_{1,2},j_{1})],[a_{2,1},(a_{2,2},j_{2})],\ldots,[a_{r,1},(a_{r,2},j_{r})]\right\rangle|}
\\&=\sum_{(j_{1},\ldots,j_{r})\in A}{|[(a_{1,2},j_{1}),(a_{2,2},j_{2}),\ldots,(a_{r,2},j_{r})]|\prod_{i=1}^{r}{\gauss{a_{i-1,2}-(2a_{i,2}-j_{i})}{a_{i,1}-(2a_{i,2}-j_{i})}}}\label{left}\tag{L},
\end{align*}
where the last two equalities follow from the fact that
$$\dim (W +TW)= 2\dim W - \dim(W \cap T^{-1}W)$$
for every subspace $W$. Then $\eqref{left} - \eqref{right} =0$. Adding $|[(a_{1,1},a_{1,2}),(a_{2,1},a_{2,2}),\ldots,(a_{r,1},a_{r,2})]|$ to both sides of this equality, we obtain the lemma.
\end{proof}



\begin{Prop}
\label{splittingsubspaces}
Let $T$ be any linear operator on an $md$-dimensional vector space $V$. Then
\begin{align*}
[((d-1)m,(d-2)m)&,((d-2)m,(d-3)m),\ldots,(2m,m),(m,0)]_{T}
\\&=\left\{(\bigoplus_{i=0}^{d-2}{T^{i}W},\ \bigoplus_{i=0}^{d-3}{T^{i}W},\ \ldots,\ W\oplus T W,\ W)\ :\  \bigoplus_{i=0}^{d-1}{T^{i}W}=V \right\}.
\end{align*}

In particular, $$\sigma(m,d;T)=|[((d-1)m,(d-2)m),((d-2)m,(d-3)m),\ldots,(2m,m),(m,0)]_T|.$$
\end{Prop}

\begin{proof}
If $W$ is an $m$-dimensional subspace of $V$ such that $\displaystyle \dim(\bigoplus_{i=0}^{d-1}{T^{i}W}) = md$, then
\begin{align*}
(\bigoplus_{i=0}^{d-2}&{T^{i}W},\bigoplus_{i=0}^{d-3}{T^{i}W},\ldots,W\oplus T W,W)
\\&\in[((d-1)m,(d-2)m),((d-2)m,(d-3)m),\ldots,(2m,m),(m,0)].
\end{align*}
Conversely, suppose 
\begin{align*}
(W_{d-1},W_{d-2},\ldots,W_2,W_{1})\in [((d-1)m,(d-2)m),((d-2)m,(d-3)m),\ldots,(2m,m),(m,0)].
\end{align*}
Let $W_0=\{0\}$ and $W_d=W_{d-1}+TW_{d-1}$. We claim that 
$$
W_n=\bigoplus_{i=1}^n T^{i-1}W_1 \quad \mbox{ for } 1\leq n\leq d.
$$
 We induct on $n$. The base case $n=1$ is evident. Now fix $1\leq k\leq d-1$ and suppose $W_j=\bigoplus_{i=1}^j T^{i-1}W_1$ for $j\leq k$. By comparing dimensions, we must have $\dim T^{i-1}W_1=m$ for $i\leq k$. We claim that $W_k\cap T^kW_1=\{0\}$. Suppose there is a nonzero vector $\beta\in W_k\cap T^kW_1$. Then $\beta=T\alpha$ for some $\alpha \in T^{k-1}W_1$ and consequently $\alpha \in W_k\cap T^{-1}W_k=W_{k-1},$ which contradicts the fact that $W_{k-1}\cap T^{k-1}W_1=\{0\}$. This proves the claim.

In fact, the restriction of $T$ to $T^{k-1}W_1$ is injective. For if $T\alpha=0$ for some nonzero $\alpha \in T^{k-1}W_1$, then $\alpha \in W_k\cap T^{-1}W_k=W_{k-1}$, contradicting the fact that $W_{k-1}\cap T^{k-1}W_1=\{0\}$. Injectivity of the restriction implies that $\dim T^kW_1=\dim T^{k-1}W_1=m.$ 

Now $W_k+TW_k\subseteq W_{k+1}$ and 
$$
  \dim (W_k+TW_k)=\dim \bigoplus_{i=1}^{k+1}T^{i-1}W_1=(k+1)m=\dim W_{k+1}.
$$
It follows that $W_{k+1}=W_k+TW_k=\bigoplus_{i=1}^{k+1}T^{i-1}W_1$, completing the proof by induction. Since $\dim W_d=\dim V$, it follows that $W_1$ is $T$-splitting.
\end{proof}

\section{Similarity Class Type and Splitting Subspaces}
\label{sec:type}

The similarity class of an operator $T$ is determined by the isomorphism type of the associated $\mb{F}_q[x]$-module on the vector space $V$ in which the action of $x$ is that of $T$. This module is isomorphic to a direct sum 
\[ \bigoplus_{i=1}^{k}\bigoplus_{j=1}^{l_{i}} 
   \mb{F}_q[x] / (p_{i}^{\lambda_{i,j}}),
\]
where $p_{1},\dots,p_{k}$ are distinct monic irreducible polynomials and, for each $1\leq i\leq k$, the sequence $\lambda_{i,1} \geq \lambda_{i,2}\geq \dots \geq \lambda_{i,l_{i}}$ is an integer partition of $n_{i} = \sum_{j}\lambda_{i,j}$. 
Let $\lambda_{i}$ denote the partition of $n_{i}$ given by the $\lambda_{i,j}$. The similarity class of $T$  is completely determined by the finite set of distinct monic irreducible polynomials 
$p_{1},\ldots,p_{k}$ and the corresponding partitions $\lambda_{1},\ldots,\lambda_{k}$.
\begin{Def}
  \label{def:sct}
If $d_i$ denotes the degree of $p_i$ for $1\leq i\leq k$ in the decomposition above, then the {\it similarity class type} of $T$ is the finite multiset $\{(d_1,\lambda_1), \ldots, (d_k, \lambda_k)\}$.  
\end{Def}
Thus, the similarity class type of $T$ keeps track of only the degrees of the polynomials and the corresponding partitions in the decomposition above. The \emph{size} of a similarity class type is the dimension of the vector space on which a corresponding operator is defined. The notion of similarity class type goes back to the work of Green \cite[p. 405]{MR72878} on the characters of the finite general linear groups.  Corollary~\ref{type} asserts that for fixed integers $m$ and $d$, the number $\sigma(m,d;T)$ depends only on the similarity class type of $T$. We begin by showing that the number of splitting subspaces depends only on the similarity class of $T$.

\begin{Prop}
Let $T$ and $T'$ be similar linear operators on an $md$-dimensional vector space $V$. Then $\sigma(m,d;T)=\sigma(m,d;T')$.
\end{Prop}
\begin{proof}
There exists a linear isomorphism $S$ of $V$ such that $T'=S\circ T\circ S^{-1}$. Then $W$ is a splitting subspace for $T$ if and only if $SW$ is a splitting subspace for $T'$. It follows that $\sigma(m,d;T)=\sigma(m,d;T')$.
\end{proof}

We begin by recalling some basic facts about lattices. A partially ordered set $P$ is called a {\it lattice} if any two elements $a,b\in P$  have a meet (greatest lower bound), denoted by $a \land b$, and a join (least upper bound), denoted $a\lor b$. We denote by $L(T)$, the set of all $T$-invariant subspaces of $V$. Clearly, $L(T)$ is a lattice with subspaces ordered by inclusion, with intersection as meet and linear sum as join. A {\it lattice homomorphism} is a mapping between lattices that preserves meets and joins. Two lattices are {\it isomorphic} if there exists a bijective lattice homomorphism between them.

Let $p=\prod_{i=1}^{k}{p_i}^{n_i}$ denote the canonical factorization of the minimal polynomial $p$ of an operator $T$ into distinct irreducible factors $p_i(1\leq i\leq k)$. Let $V_i = \{ \alpha \in V : {p_i}(T)^{n_i}\alpha=0 \}$. Then $V_i$ is a $T$-invariant subspace of $V$ and
\begin{align*}
V=\bigoplus_{i=1}^{k}V_i.
\end{align*}
This is the primary decomposition of $V$. We call an operator {\it primary} or {\it $p$-primary} if its minimal polynomial is a power of the irreducible polynomial $p$. Denote by $T_i$ the restriction of $T$ to $V_i$. Then $T_i$ is a linear operator on $V_i$. It is known \cite[Thm. 1]{MR213378} that the lattice $L(T)$ is the direct product of the lattices $L(T_i)$, i.e.,  
\begin{align*}  
L(T) = \prod_{i=1}^{k} L(T_i).
\end{align*}
Thus, for each $U \in L(T)$, there exists precisely one sequence $(U_1, \ldots, U_k) \in \prod_{i=1}^{k} L(T_i)$ such that $U = U_1 \oplus \cdots \oplus U_k$. Consequently, it suffices to study the lattices $L(T_i)$ corresponding to the primary components $V_i, 1 \leq i \leq k$.

\begin{Prop}
\label{inv}
If $T$ and $T'$ are similar then there exists a dimension preserving isomorphism between $L(T)$ and $L(T')$.
\end{Prop}

\begin{proof}
Let $S$ be an invertible transformation such that $T'=S\circ T\circ S^{-1}$. Then $W$ is $T$-invariant if and only if $SW$ is $T'$-invariant. Therefore the map $W \mapsto SW$ is a dimension preserving lattice isomorphism from $L(T)$ to $L(T')$. 
\end{proof}

\begin{Thm} \label{thmbrick}
Let $T$ be a linear  operator on a vector space $V$ over a field $F$ such that $T$ is $p$-primary with $p$ separable. Let $T=S+Q$ denote the Jordan-Chevalley decomposition of $T$ into its semi-simple part $S$ and nilpotent part $Q$. Let $K$ be the algebra of polynomials in $S$ over $F$. Then $K$ is a field isomorphic to $ F[x]/(p(x))$, $V$ is naturally a $K$-vector space, $T$ is $K$-linear, and $L_{F}(T)=L_{K}(T)=L_{K}(Q)$. 
\end{Thm}

\begin{proof}
See \cite[Thm. 6]{MR213378}.
\end{proof}

The above theorem applies to every primary operator defined over a finite field since irreducible polynomials in this setting are separable.

\begin{Thm} \label{lattice}
Let $T$ and $T'$ be linear operators of the same similarity class type defined on a vector space $V$ over $\Fq$. Then there exists a dimension preserving isomorphism between $L(T)$ and $L(T')$.
\end{Thm}

\begin{proof}
It suffices to prove the result when $T$ and $T'$ are primary operators. Let $T$ be $p$-primary and $T'$ be $p'$-primary. Let $T=S+Q$ and $T'=S'+Q'$ where $S$ and $S'$ are semi-simple while $Q$ and $Q'$ are nilpotent. Further, let $K$ and $K'$ be the algebras of polynomials in $S$ and $S'$ respectively over $\mb{F}_q$. Theorem~\ref{thmbrick} implies that the fields $K$ and $K'$ are isomorphic since $p$ and $p'$ are irreducible polynomials over $\mb{F}_q$ of the same degree. Further, $L_{\mb{F}_q}(T)=L_{K}(Q)$ and $L_{\mb{F}_q}(T')=L_{K'}(Q')$. Since $T$ and $T'$ are of same similarity class type, it follows that their nilpotent parts $Q$ and $Q'$ are also of the same type. Thus $Q$ and $Q'$ are similar. By Proposition~\ref{inv} and the fact that $K \cong K'$, we obtain a dimension preserving isomorphism between $L_{K}(Q)$  and $L_{K'}(Q')$. 
\end{proof}

\begin{Rem}
In light of the above theorem, given $q$, we may define the number of invariant subspaces of dimension $k$ of a similarity class type $\tau$ to be the number of invariant subspaces of dimension $k$ of some operator $T$ of type $\tau$ over $\Fq$.
\end{Rem}

\begin{Cor} \label{type}
 Suppose $T$ and $T'$ are two operators of the same similarity class type defined on an $md$-dimensional vector space over $\Fq$.Then $\sigma(m,d;T)=\sigma(m,d;T').$
\end{Cor}

\begin{proof}
  The sets $ [(a_{1},a_{1}),(a_{2},a_{2}),\ldots,(a_{r},a_{r})]_T$ corresponding to base cases  in the recursion of Lemma~\ref{recursion} consist of flags of invariant subspaces $(W_1, \ldots, W_r)$ such that $\dim W_i=a_i$ and $W_{i} \supseteq W_{i+1}$ for $1\leq i\leq r-1$. The existence of a dimension preserving isomorphism between $L(T)$ and $L(T')$ ensures that the base cases coincide:
  $$ |[(a_{1},a_{1}),(a_{2},a_{2}),\ldots,(a_{r},a_{r})]_T| =  |[(a_{1},a_{1}),(a_{2},a_{2}),\ldots,(a_{r},a_{r})]_{T'}|.$$
  Therefore 
\begin{align}
\left|[(a_{1,1},a_{1,2}),(a_{2,1},a_{2,2}),\ldots,(a_{r,1},a_{r,2})]_T\right|=\left|[(a_{1,1},a_{1,2}),(a_{2,1},a_{2,2}),\ldots,(a_{r,1},a_{r,2})]_{T'}\right|, \label{eqn2}
\end{align}
whenever the two quantities are defined. In particular, by Proposition~\ref{splittingsubspaces}, we must have $\sigma(m,d;T)=\sigma(m,d;T').$
\end{proof}


\begin{Cor}
\label{cor:scalar}
Let $T$ be a linear operator on an $md$-dimensional vector space $V$ over $\Fq$ and suppose $c\in \Fq$. If $I$ denotes the identity on $V$, then $\sigma(m,d;T)=\sigma(m,d;T+cI).$
\end{Cor} 

\begin{proof}
 This follows from Corollary \ref{type} as $T$ and $T+cI$ have the same similarity class type.
\end{proof}

\section{Cyclic Nilpotent Operators} 
\label{sec:cyclic}
 In this section, we determine the number of splitting subspaces of a cyclic nilpotent operator by guessing a formula that satisfies the recursion in Lemma \ref{recursion}. The following proposition will prove useful in the computation of base cases.  

\begin{Prop}
  \label{pr:chain}
Let $T$ be a cyclic $p$-primary operator on a vector space $U$ of dimension $ad$ where $d=\deg p$. Then
\begin{align*}
L(T)= \left \{ \ker\ p(T)^j  : 0\leq j\leq a \right\}.    
\end{align*}
\end{Prop}

\begin{proof}
See \cite[Lem. 2]{MR213378}. 
\end{proof}

Suppose $T$ is a cyclic nilpotent operator on $V$ and $\dim V=N$. By Proposition~\ref{pr:chain}, there is precisely one $T$-invariant subspace of dimension $k$ for each integer $0\leq k\leq \dim V$, namely, the kernel of $T^{k}$. As restrictions of cyclic operators to invariant subspaces are cyclic as well, it follows that  $$ \displaystyle |[(a_{1},a_{1}),(a_{2},a_{2}),\ldots,(a_{r},a_{r})]_{T}| =1 \qquad  ( N \geq a_1 \geq a_2 \geq \ldots \geq a_r \geq 0).$$

We require a few lemmas before we proceed to solve the recursion. In what follows, the notation $\sum_{s}$ signifies a sum taken as $s$ varies over all integers with the convention that the $q$-binomial coefficient ${n \brack k}_q$ is zero whenever either $n$ or $k$ is negative, or when $k$ does not lie between $0$ and $n$.
\begin{Lem} \label{id}
For integers $a,b,c$, we have \cite[p. 42]{MR2339282}  
\begin{align*}
\gauss{a}{b} \gauss{b}{c} = \gauss{a}{c} \gauss{a-c}{b-c}.
\end{align*}
\end{Lem}

\begin{Lem}
For nonnegative integers $a,b,r$, the $q$-Vandermonde identity \cite[Thm. 3.4]{Andrews} holds:  
  \begin{align*}
    {a+b \brack r}_q =\mathlarger\sum_s {a \brack s}_q{b \brack r-s}_q q^{s(b-r+s)}.
  \end{align*}
\end{Lem}

\begin{Lem} \label{first}
For nonnegative integers $a\geq d\geq b\geq c$, we have
  $$
\sum_{s} {a-b \brack b-s}_q {b-c \brack s-c}_q {a-2b+s \brack d-2b+s}_q q^{(b-s)^2}={a-b \brack d-b}_q {d-c \brack b-c}_q.
  $$
\end{Lem}
\begin{proof}
By the $q$-Vandermonde identity, 
\begin{align*}
  \sum_{s}{d \brack d-s}_q {b-c \brack s}_q q^{s^2}&={d+b-c \brack d}_q.\\
  \therefore  \sum_{s}{a \brack d}_q {d \brack s}_q {b-c \brack s}_q q^{s^2}&={a \brack d}_q {d+b-c \brack d}_q.
\end{align*}                                                                              
Now apply Lemma~\ref{id} and replace $s$ by $b-s$ to obtain
\begin{align*}
  \sum_{s}{a \brack s}_q {a-s \brack d-s}_q {b-c \brack s}_q q^{s^2}&={a \brack d}_q {d+b-c \brack d}_q.\\
  \therefore \sum_{s}{a \brack b-s}_q {a-b+s \brack d-b+s}_q {b-c \brack s-c}_q q^{(b-s)^2}&={a \brack d}_q {d+b-c \brack d}_q    .
\end{align*}
Replacing $a$ by $a-b$ and $d$ by $d-b$ results in the statement of the lemma.
\end{proof}

\begin{Lem} \label{second}
For nonnegative integers $a\geq b\geq  d\geq c$,
  $$
\sum_{s} {a-b \brack b-s}_q {b-c \brack s-c}_q {s-c \brack d-c}_q q^{(b-s)^2}={b-c \brack d-c}_q {a-d \brack b-d}_q.
  $$
\end{Lem}
\begin{proof}
  By the $q$-Vandermonde identity, we have
  \begin{align*}
    \sum_{s}{a-b \brack b-s}_q {b-d \brack s-d}_q q^{(b-s)^2}&={a-d \brack b-d}_q.\\
\therefore    \sum_{s}{a-b \brack b-s}_q {b-c \brack d-c}_q {b-d \brack s-d}_q q^{(b-s)^2}&={b-c \brack d-c}_q {a-d \brack b-d}_q.
  \end{align*}
  The result follows from Lemma~\ref{id} since
  \begin{align*}
{b-c \brack d-c}_q {b-d \brack s-d}_q={b-c \brack s-c}_q {s-c \brack d-c}_q.    
  \end{align*}\qedhere
\end{proof}

We now solve the recursion stated in Lemma~\ref{recursion} for a cyclic nilpotent operator.
\begin{Thm}\label{alpha}
Let $T$ be a cyclic nilpotent operator on $V$ with $\dim V=N$ and suppose  
\begin{align*}
N\ge a_{1,1}\ge a_{1,2}\ge a_{2,1}\ge a_{2,2}\ge \cdots \ge a_{r,1}\ge a_{r,2}\ge 0,
\\a_{0,1}=a_{0,2}=N, \quad a_{r+1,1}=a_{r+1,2}=0.
\end{align*} 
Then
\begin{align}
|[(a_{1,1},a_{1,2}),\ldots, (a_{r,1},a_{r,2})]|=
 \prod_{i=1}^{r} 
\gauss{a_{i-1,1}-a_{i,1}}{a_{i,1}-a_{i,2}}
\gauss{a_{i,1}-a_{i+1,1}}{a_{i,2}-a_{i+1,1}}
q^{{(a_{i,1}-a_{i,2})}^2}.\label{formula}
\end{align}
\end{Thm}
 
\begin{proof}[Proof]
We show that the formula stated above satisfies the recursion of Lemma~\ref{recursion} by computing separately the sums over the sets $A$ and $B$ defined there. Substitute the expression for $|[(a_{1,1},a_{1,2}),\ldots, (a_{r,1},a_{r,2})]|$ given by \eqref{formula} into the recursion to obtain
\begin{align*}
L &=\sum_{(j_{1},\ldots,j_{r})\in A}{|[(a_{1,2},j_{1}),(a_{2,2},j_{2}),\ldots,(a_{r,2},j_{r})]|\prod_{i=1}^{r}{\gauss{a_{i-1,2}-(2a_{i,2}-j_{i})}{a_{i,1}-(2a_{i,2}-j_{i})}}}\\ 
&=\prod_{i=1}^{r}{\sum_{j_{i}}{\gauss{a_{i-1,2}-a_{i,2}}{a_{i,2}-j_{i}}
\gauss{a_{i,2}-a_{i+1,2}}{j_{i}-a_{i+1,2}}
\gauss{a_{i-1,2}-(2a_{i,2}-j_{i})}{a_{i,1}-(2a_{i,2}-j_{i})}q^{{(a_{i,2}-j_{i})}^2}}}.
\end{align*}
Apply Lemma~\ref{first} to each sum in the above expression, followed by Lemma \ref{id} to obtain
\begin{align*}
L=
\prod_{i=1}^r 
\gauss{a_{i-1,2}-a_{i,2}}{a_{i,1}-a_{i,2}}
\gauss{a_{i,1}-a_{i+1,2}}{a_{i,2}-a_{i+1,2}}
&=
\prod_{i=1}^{r}\frac{\gauss{a_{i-1,2}}{a_{i,1}}
\gauss{a_{i,1}}{a_{i,2}}}{\gauss{a_{i-1,2}}{a_{i,2}}}
\frac{\gauss{a_{i,1}}{a_{i,2}}
\gauss{a_{i,2}}{a_{i+1,2}}}{\gauss{a_{i,1}}{a_{i+1,2}}}\\
&=
\frac{1}{\gauss{N}{a_{1,2}}}
\prod_{i=1}^{r}\frac{\gauss{a_{i,1}}{a_{i,2}}^2
\gauss{a_{i-1,2}}{a_{i,1}}}{\gauss{a_{i,1}}{a_{i+1,2}}}.
\end{align*}

On the other hand,
\begin{align*}
R &=\sum_{(k_{1},\ldots,k_{r})\in B} {|[(a_{1,1},k_{1}),(a_{2,1},k_{2}),\ldots,(a_{r,1},k_{r})]|\prod_{i=1}^{r}{\gauss{k_{i}-a_{i+1,1}}{a_{i,2}-a_{i+1,1}}}}\\
&=\prod_{i=1}^{r}{\sum_{k_{i}}{\gauss{a_{i-1,1}-a_{i,1}}{a_{i,1}-k_{i}}
\gauss{a_{i,1}-a_{i+1,1}}{k_{i}-a_{i+1,1}}
\gauss{k_{i}-a_{i+1,1}}{a_{i,2}-a_{i+1,1}}q^{{(a_{i,1}-k_{i})}^2}}}.
\end{align*}
Apply Lemma~\ref{second} to each sum in the above expression, followed by Lemma \ref{id} to obtain

\begin{align*}
R=
\prod_{i=1}^r 
\gauss{a_{i,1}-a_{i+1,1}}{a_{i,2}-a_{i+1,1}}
\gauss{a_{i-1,1}-a_{i,2}}{a_{i,1}-a_{i,2}}
&=
\prod_{i=1}^{r}\frac{\gauss{a_{i,1}}{a_{i,2}}
\gauss{a_{i,2}}{a_{i+1,1}}}{\gauss{a_{i,1}}{a_{i+1,1}}}
\frac{\gauss{a_{i-1,1}}{a_{i,1}}
\gauss{a_{i,1}}{a_{i,2}}}{\gauss{a_{i-1,1}}{a_{i,2}}}\\
&=
\gauss{N}{a_{1,1}}
\prod_{i=1}^{r}\frac{\gauss{a_{i,1}}{a_{i,2}}^2
\gauss{a_{i,2}}{a_{i+1,1}}}{\gauss{a_{i-1,1}}{a_{i,2}}}.
\end{align*}
Therefore
\begin{align*}
 \frac{L}{R} &= \frac{1}{\gauss{N}{a_{1,2}}\gauss{N}{a_{1,1}}}\gauss{N}{a_{1,1}}\gauss{N}{a_{1,2}}=1. 
\end{align*}

Hence $L=R$, proving that the given expression satisfies the recurrence. The expression in \eqref{formula} satisfies the base cases since 
\begin{align*}
|[(a_{1},a_{1}),(a_{2},a_{2}),\ldots, (a_{r},a_{r})]|
 &=
\prod_{i=1}^{r} 
\gauss{a_{i-1}-a_{i}}{0}
\gauss{a_{i}-a_{i+1}}{a_{i}-a_{i+1}}
q^{(a_i-a_i)^2}\\
&=1.
\end{align*}
This completes the proof.
\end{proof}

The following result gives the number of splitting subspaces for cyclic nilpotent operators.
\begin{Cor}
\label{grssc}
When $N\geq md$, we have the equality
\begin{align*}
|[((d-1)m,(d-2)m),\ldots, (2m,m), (m,0)]|&=\gauss{N-md+m}{m}q^{m^2 (d-1)}.
\end{align*}
In particular, when $N=md$,
\begin{align*}
|[((d-1)m,(d-2)m),\ldots, (2m,m), (m,0)]|
= q^{m^2 (d-1)}.
\end{align*}
\end{Cor}

\begin{proof}
By Theorem \ref{alpha}, we have
\begin{align*}
&|[((d-1)m,(d-2)m),((d-2)m,(d-3)m),\ldots,(2m,m),(m,0)]|\\
  &=\gauss{N-(d-1)m}{m} q^{m^2} \prod_{i=2}^{d-1} 
\left(\gauss{m}{m} \gauss{m}{0} q^{m^2}\right)
\\&=\gauss{N-md+m}{m}q^{m^2(d-1)}. \qedhere
\end{align*}
\end{proof}
Combining the above corollary with Corollary \ref{cor:scalar}, we obtain the following result. 
\begin{Cor}
  \label{cor:1primary}
Let $T$ be a linear operator on an $md$-dimensional vector space over $\Fq$. If $T$ is cyclic and $p$-primary for some linear polynomial $p$, then $\sigma(m,d;T)=q^{m^2(d-1)}$.
\end{Cor}
\begin{Rem}
One of the main results of Chen and Tseng \cite[Cor. 3.4]{sscffa} is the computation of $\sigma(m,d;T)$ when the similarity class type of $T$ is $\{\left(md,(1)\right)\}$. The above corollary corresponds to the similarity class type $\{\left(1,(md)\right)\}$.  
\end{Rem}

Our results may be used to enumerate invertible matrices of a special form. Recall that the number of ordered bases for an $m$-dimensional vector space over $\mb{F}_q$ is given by $\gamma_m(q) :=(q^m-1) \ldots (q^m-q^{m-1}).$
\begin{Cor}
The number of invertible $md \times md$ matrices over $\mb{F}_q$ of the form
\begin{equation} 
\left [
\begin{array}{cccccccccc}\label{matrix} 
a_{1,1} & \cdots & a_{1,m} & 0 & \cdots & 0 & \cdots & 0 & \cdots & 0 \\
a_{2,1} & \cdots & a_{2,m} & a_{1,1} & \cdots & a_{1,m} &\cdots & 0 & \cdots & 0 \\
\vdots & \vdots & \vdots & \vdots & \vdots & \vdots & \vdots & \vdots & \vdots & \vdots \\
a_{d,1} & \cdots & a_{d,m} & a_{d-1,1} & \cdots & a_{d-1,m} & \cdots & a_{1,1}& \cdots & a_{1,m} \\
\vdots & \vdots & \vdots & \vdots & \vdots & \vdots & \vdots & \vdots & \vdots & \vdots \\
a_{md,1} & \cdots & a_{md,m} & a_{md-1,1} & \cdots & a_{md-1,m} & \cdots & a_{md-d+1,1}& \cdots & a_{md-d+1,m} \\
\end{array} 
\right ]
\end{equation}
equals $\gamma_m(q)\cdot q^{m^2(d-1)}$.
\end{Cor}
\begin{proof}
  Let $V=\Fq^{md}$ viewed as a vector space over $\Fq$. Let $T$ denote the right shift operator
$
T(x_1,\ldots,x_{md})=(0,x_1,\ldots,x_{md-1}).
$
Then $T$ is cyclic and nilpotent. If $\alpha_1,\ldots,\alpha_m$ denote the first $m$ columns of the matrix \eqref{matrix} above, then the matrix is nonsingular if and only if the set
$$
\{\alpha_1,\ldots,\alpha_m,T\alpha_1,\ldots,T\alpha_m,\ldots,T^{d-1}\alpha_1,\ldots,T^{d-1}\alpha_m\}
$$
is linearly independent. In other words $\{\alpha_1,\ldots,\alpha_m\}$ may be characterized as an ordered basis for some $m$-dimensional $T$-splitting subspace. Since number of such bases is $\gamma_m(q)\cdot \sigma(m,d;T)$, the result now follows from Corollary \ref{cor:1primary}.   
\end{proof}

A result in a similar vein to the above corollary is proved by Gluesing-Luerssen and Ravagnani \cite[Cor. 7.2]{MR4093016} who find an expression for the number of nonsingular matrices over $\F_q$ whose nonzero entries lie within a Ferrers shape. 
  
\section{Polynomiality of $\sigma_q(m,d;\tau)$}
\label{sec:poly}
Given a positive integer $n$, let $\beta(q,n)$ denote the number of irreducible polynomials of degree $n$ over $\Fq$. It is well known that
$$
\beta(q,n)=\frac{1}{n}\sum_{d\mid n}\mu(d)q^{n/d},
$$
where $\mu$ denotes the classical Möbius function. In fact $\beta(n,d)$ also counts the number of so called primitive necklaces of length $n$ over a $q$-ary alphabet \cite[Thm. 7.1]{MR1231799}. This interpretation entails the fact that, for $n$ fixed, the number $\beta(q,n)$ is strictly increasing as a function of $q$.

Given a similarity class type $\tau$, let $q_0(\tau)$ denote the smallest prime power $\tilde{q}$ for which there exists a linear operator of type $\tau$ over the field $\mathbb F_{\tilde{q}}$. If $q\geq q_0(\tau)$ is a prime power then it follows from the property of $\beta(q,n)$ mentioned above that there exists a linear operator of type $\tau$ over $\Fq$.
\begin{Def}
If $\tau$ is a similarity class type of size $md$ and $q\geq q_0(\tau)$, then $\sigma_q(m,d;\tau)$ denotes the number of $m$-dimensional splitting subspaces for a linear operator of similarity class type $\tau$ defined over an $\Fq$-vector space of dimension $md$.  
\end{Def}
 By Corollary~\ref{type}, the quantity $\sigma_q(m,d;\tau)$ is well-defined. We will prove that $\sigma_q(m,d;\tau)$ is a polynomial in $q$. 
In view of \eqref{eqn2}, given a similarity class type $\tau$, and a prime power $q\geq q_0(\tau)$,  we may define
$$
\left|[(a_{1,1},a_{1,2}),(a_{2,1},a_{2,2}),\ldots,(a_{r,1},a_{r,2})]_{q,\tau}\right|:= \left|[(a_{1,1},a_{1,2}),(a_{2,1},a_{2,2}),\ldots,(a_{r,1},a_{r,2})]_T\right|,
$$
where $T$ is an operator of similarity class type $\tau$ defined over some $\Fq$ vector space. 

 \begin{Thm}
   \label{th:polynomiality}
   For each similarity class type $\tau$ and nonnegative integers $a_{i,j}(1\leq i\leq r; 1\leq j\leq 2)$, the quantity
   $$
\left|[(a_{1,1},a_{1,2}),\ldots, (a_{r,1},a_{r,2})]_{q,\tau}\right|
$$
is a polynomial in $q$ for $q\geq q_0(\tau)$.
\end{Thm}
\begin{proof}

As the $q$-binomial coefficients are polynomials in $q$, it suffices to show that the base cases in the recursion of Lemma \ref{recursion} are polynomials in $q$ for the similarity class type $\tau$. For any operator $T$, define
\begin{align*}
\phi(a_1,\ldots,a_r;T):=|[(a_1,a_1),\ldots,(a_r,a_r)]_T|,
\end{align*}
and, for each similarity class type $\tau$, let
\begin{align*}
\phi_q(a_1,\ldots,a_r;\tau):=|[(a_1,a_1),\ldots,(a_r,a_r)]_{q,\tau}|\quad \mbox{ for }q\geq q_0(\tau).
\end{align*}
 We first prove that $\phi_q(a_1,\ldots,a_r;\tau)$ is a polynomial in $q$ whenever $\tau$ is primary by induction on $r$. Suppose $\tau=\{(d,\lambda)\}$ for some partition $\lambda$. The base case is $r=1$. Now $\phi_q(a_1;\tau)$ is the number of invariant subspaces of dimension $a_1$ corresponding to $\tau$.  
 It follows from the work of Fripertinger \cite[Thm. 2]{MR2801603} that the number of invariant subspaces of a given dimension for a primary similarity class type on an $\Fq$-vector space is a rational function of $q$. Such a function is necessarily a polynomial in $q$ since it takes integer values at infinitely many integers \cite[Prop. X.1.1]{MR1421321}. This settles the base case.


For the inductive step suppose $r>1$. For each prime power $q$, let $T_q$ be an operator of type $\tau$ defined over $\Fq$. For each partition $\mu$ whose Young diagram is contained in that of $\lambda$ (denoted $\mu\subseteq \lambda$), let $g_q(\lambda,\mu,d)$ denote the number of subspaces $W \in L(T_q)$ for which the similarity class type of the restriction of $T_q$ to $W$ is $\{(d,\mu) \}.$ For $\lambda$ and $\mu$ fixed, the quantity $g_q(\lambda,\mu,d)$ is a rational function \cite[Thm. 1]{MR2801603} and, in fact, a polynomial in $q^d$. For each positive integer $k$, define 
$$
D_\tau(k): = \{ \mu : \mu \subseteq \lambda \mbox{ and } |\mu|=k/d \},
$$
and set $\delta_\tau(k) = |D_\tau(k)|$. If we write $D_\tau(a_1)=\{\mu_i\}_{1\leq i\leq \delta_\tau(a_1)}$ then, by considering restrictions of $T_q$ to invariant subspaces of dimension $a_1$, it follows that 
$$
\phi(a_1,\ldots,a_{r};T_q)= \sum_{j=1}^{\delta_\tau(a_1)} g_q(\lambda,\mu_j,d) \ \phi(a_2,\ldots,a_{r};T_q^{(j)}),
$$
where $T_q^{(j)}$ is an operator of similarity class type $\{(d,\mu_j)\}$ for each $j\leq \delta_{\tau}(a_1)$. Since the above equation holds for each prime power $q$, it follows that 
$$
\phi_q(a_1,\ldots,a_{r};\tau)= \sum_{j=1}^{\delta_\tau(a_1)} g_q(\lambda,\mu_j,d) \ \phi_q(a_2,\ldots,a_{r};\tau_j),
$$
where $\tau_j = \{(d, \mu_j)\}$.  By the inductive hypothesis, each $\phi_q(a_2,\ldots,a_{r};\tau_j)$ is a polynomial in $q$. Therefore $\phi_q(a_1,\ldots,a_{r};\tau)$ is a polynomial in $q$ whenever $\tau$ is primary.

Now suppose $\tau$ is an arbitrary similarity class type and let $q\geq q_0(\tau)$ be a prime power. Let $T(=T_q)$ be an operator of similarity class type $\tau$ defined on some vector space $V$ over $\Fq$. Let $V=V_1 \oplus \cdots \oplus V_s$ denote the decomposition of $V$ into primary parts $V_i$ with $\dim V_i = d_i$ for each $i$. Suppose $T_i$ denotes the restriction of $T$ to $V_i$ for $i\leq r$. Given any flag $W_1 \supseteq \cdots \supseteq W_r$ in $L(T)$, with $\dim W_i = a_i$, 
write $W_i=U_{i1} \oplus \cdots \oplus U_{is}$  with $U_{ij}\in V_j$ for $1\leq i\leq r$ and $1\leq j\leq s$.
If $\dim U_{ij}=d_{ij}$, 
then it follows that
\begin{align*}
a_i = \sum_{j=1}^{s} d_{ij} & \ (1 \leq i \leq r)\quad  \mbox{ and }\quad  d_j\geq d_{1j}\geq \cdots \geq d_{rj} \ ( 1 \leq j \leq s).
\end{align*} 
Counting flags within the primary parts and summing up, we obtain
$$
\phi(a_1,\ldots,a_{r};T)= \sum_{\substack{\sum_{j=1}^{s}d_{ij}=a_i \\ d_{1j} \leq d_j}} \ \prod_{j=1}^{s} \phi(d_{1j},\ldots,d_{rj};T_j),
$$
where the $d_{ij}$'s vary over nonnegative integers. It follows that
$$
\phi_q(a_1,\ldots,a_{r};\tau)= \sum_{\substack{\sum_{j=1}^{s}d_{ij}=a_i \\ d_{1j} \leq d_j}} \ \prod_{j=1}^{s} \phi_q(d_{1j},\ldots,d_{rj};\tau_j),
$$
whenever $q\geq q_0(\tau)$ with $\tau_j$ denoting the similarity class type of $T_j \ (1 \leq j \leq s).$ Since $\tau_j$ is primary for each $j\leq s$, the expression $\phi_q(d_{1j},\ldots,d_{rj};\tau_j)$ is a polynomial in $q$ and, consequently, so is $\phi_q(a_1,\ldots,a_{r};\tau)$.
\end{proof}

\begin{Cor}
If $m,d,\tau$ are fixed, then $\sigma_q(m,d;\tau)$ is a polynomial in $q$.
\end{Cor} 
\begin{proof}
  Follows from Theorem \ref{th:polynomiality} and Proposition \ref{splittingsubspaces}. 
\end{proof} 

\begin{Rem}
The polynomial $g_q(\lambda,\mu,d)$ appearing in the proof of the theorem above is closely related to the number of subgroups of type $\mu$ in a finite abelian $p$-group of type $\lambda$, denoted $\alpha_{\lambda}(\mu;p)$.  Delsarte~\cite{MR25463} proved that $\alpha_{\lambda}(\mu;p)$ is a polynomial in $p$. We refer to the expository account of Butler \cite[Lem. 1.4.1]{MR1223236} for the details. It can be shown that $g_q(\lambda,\mu,d)=\alpha_{\lambda}(\mu;q^d)$. 
\end{Rem}
\section{Acknowledgements}   
The authors would like to thank Dennis Tseng for carefully going through the initial draft of this paper and suggesting improvements. They extend thanks to Amritanshu Prasad for some helpful discussions.

\end{document}